\documentclass{sl3}     

\usepackage{slsec}     
\usepackage{stud_log} 
\usepackage{slfoot2}
\usepackage{slthm}     
                         
\usepackage{amsmath}
\usepackage{amssymb}
\usepackage{csquotes}
\usepackage{setspace}

\usepackage{titling}
\usepackage{url,color,amsfonts}
\usepackage{authblk,abstract}
\usepackage{comment}
\usepackage{tikz}
\usepackage{mathtools}

\usepackage{enumerate}

\def\kn{\kern.1em}

\newcommand{\meet}{\wedge}
\newcommand{\join}{\vee}

\newcommand{\up}{\mathord{\uparrow}}

\newcommand{\m}{\mathbf}
\newcommand{\gbl}{GBL}
\newcommand{\var}{{\sf Var}}
\newcommand{\km}{\Vdash}

\newcommand{\mv}{[0,1]_{\textrm{MV}}}

\newcommand{\eq}{\approx}

\DeclarePairedDelimiter\floor{\lfloor}{\rfloor}

\renewcommand{\theequation}{\Roman{equation}}

\newtheorem{theorem}{Theorem}[section]
\newtheorem{lemma}[theorem]{Lemma}
\newtheorem{proposition}[theorem]{Proposition}

\theoremstyle{definition}
\newtheorem{definition}[theorem]{Definition}

\newtheorem{remark}[theorem]{Remark}

\newcommand\blfootnote[1]{%
  \begingroup
  \renewcommand\thefootnote{}\footnote{#1}%
  \addtocounter{footnote}{-1}%
  \endgroup
}

\HeadingsInfo{Wesley Fussner}{Poset products as relational models}

\begin{document}

\setcounter{page}{1}     

 





\renewcommand{\footnoterule}{\noindent\rule{5cm}{0.4pt}{\vspace{5pt}}}%
\renewcommand\thefootnote{\arabic{footnote}}%

\begin{center}
{\Large\bf Poset products as relational models}\\
\vspace{6pt}
Wesley Fussner\blfootnote{This project received funding from the European Research Council (ERC) under the European Union’s Horizon 2020 research and innovation program (grant agreement No. 670624).}\\
\vspace{12pt}
{\footnotesize Laboratoire J.A. Dieudonn\'e, CNRS, and Universit\'e C\^ote d'Azur, France}
\vspace{6pt}
\end{center}
   

\setcounter{footnote}{0}%


{\setstretch{0.85}
\begin{displayquote}
{\footnotesize {\bf Abstract.} We introduce a relational semantics based on poset products, and provide sufficient conditions guaranteeing its soundness and completeness for various substructural logics. We also demonstrate that our relational semantics unifies and generalizes two semantics already appearing in the literature: Aguzzoli, Bianchi, and Marra's temporal flow semantics for H\'ajek's basic logic, and Lewis-Smith, Oliva, and Robinson's semantics for intuitionistic \L ukasiewicz logic. As a consequence of our general theory, we recover the soundness and completeness results of these prior studies in a uniform fashion, and extend them to infinitely-many other substructural logics.}
\end{displayquote}}



\section{Introduction} 

Kripke-style relational semantics has proven wildly successful in the study of intuitionistic logic and its extensions, and represents one of the dominant viewpoints on these logics. Substructural logics comprise a diverse family of non-classical logics generalizing intuitionistic logic, and insights from intuitionistic logic have inspired numerous fruitful lines of inquiry in the substructural setting. However, relational semantics has won comparatively few victories in the substructural case, especially in comparison to the prodigious success of algebraic semantics. For instance, the Routley-Meyer semantics \cite{RM1973}---probably the best known relational semantics for a properly non-intuitionistic substructural logic---is based on ternary accessibility relations rather than binary ones. These are demanding to work with on a technical level, and admit several competing conceptual interpretations (see \cite{M2020} for a survey). This extra layer of complexity has been an obstacle for the widespread application of the Routley-Meyer semantics and its cognates. An easily-applicable, conceptually-simple substructural relational semantics has so far eluded capture except in a few special cases (see, e.g., \cite{D1976,FG2019}). 

The purpose of this paper is to show how Jipsen and Montagna's poset product construction (see, e.g., \cite{JM2010}) may be interpreted as a relational semantics for certain propositional substructural logics, mainly those neighboring H\'ajek's basic fuzzy logic \cite{H1998}. Our discussion is motivated by \cite{ABM2009} and \cite{LOR2020}, which provide relational semantics for H\'ajek's basic logic and intuitionistic \L ukasiewicz logic ({\bf I\L L}), respectively. Basic logic is H\'ajek's effort to provide a minimal system encompassing fuzzy logics. On the other hand, {\bf I\L L} amounts to a natural deduction system for the logic of bounded commutative generalized basic logic algebras (see, e.g., \cite{JM2006}), a common generalization of basic logic and intuitionistic propositional logic (see \cite{GJ2009} for an algebraically-minded survey). At a glance, the methods of \cite{ABM2009} and \cite{LOR2020} appear quite different. Although motivated by algebraic work on poset products, \cite{LOR2020} only skirts algebraic methods. On the other hand, \cite{ABM2009} is inspired by temporal logic and engages heavily with algebraic methods, but does not connect with poset products (the theory of which was still nascent at the time of \cite{ABM2009}'s publication).

The present inquiry makes full use of the algebraic theory of poset products, and thereby unifies \cite{ABM2009} and \cite{LOR2020}. Our unified treatment is simpler than its antecedents, while also dramatically generalizing their results. In particular, we obtain general sufficient conditions (Lemma~\ref{thm:general condition}) that guarantee the soundness and completeness of our poset-product-based relational semantics in many situations. We recover the main soundness and completeness results of \cite{ABM2009,LOR2020} as an immediate consequence of this general theory, and generalize them to infinitely-many other substructural logics (see Theorem~\ref{thm:zoo}). Moreover, we provide some early results on the analogue of Kripke completeness.

Although our results represent a significant advancement from the existing literature, our methods are mundane. The main contributions of this paper consist of properly phrasing and contextualizing key ideas, definitions, and results, after which proofs become a straightforward application of known algebraic notions. Consequently, this article has an expository dimension. We thus make every effort to be as thorough as possible in our discussion of background material, and try to provide an especially explicit guide to the related literature.

The paper is structured as follows. Section \ref{sec:algebraic preliminaries} gives needed background on residuated lattices and algebraic logic. Section \ref{sec:poset products} provides a brief introduction to poset products, and synthesizes the literature on poset products for our present application. Section \ref{sec:frames} defines the relational models that we study. It also proves some general results regarding these, and compares this material to the literature. Section~\ref{sec:soundness and completeness} provides the main soundness and completeness results for our relational semantics. Finally, Section \ref{sec:afmp} provides some comments on when logics may be characterized by our relational models, and relates this to the substructural hierarchy.

\section{Algebraic preliminaries}\label{sec:algebraic preliminaries}

This preliminary section reviews the most important algebraic background. We assume familiarity with elementary lattice theory and universal algebra, and refer the reader to \cite{BS1981} for more information. 

\subsection{Residuated lattices.}\label{sec:residuated lattices} An algebra ${\m A} = (A,\meet,\join,\cdot,\to,0,1)$ is a \emph{bounded commutative integral residuated lattice} provided that:
\begin{enumerate}[\normalfont (1)]
\item $(A,\meet,\join,0,1)$ is a bounded lattice.
\item $(A,\cdot,1)$ is a commutative monoid.
\item For all $x,y,z\in A$,
$$x\cdot y\leq z\iff x\leq y\to z.$$
\end{enumerate}
We call the above-defined structures \emph{residuated lattices} for brevity.\footnote{Readers familiar with residuated structures will notice that many of our comments and results apply beyond the bounded, commutative, and integral case, but we opt to work in this setting for clarity and simplicity.} Residuated lattices are the subject of a large and rapidly-expanding literature, fueled to a great extent by their ties to substructural logic (see Section \ref{sec:algebraic substructural logic}). We invite the reader to consult the standard monograph \cite{GJKO2007} for an extensive treatment of residuated lattices and their relevance to substructural logic. Note that some sources denote multiplication $\cdot$ in residuated lattices by $\odot$ or $\otimes$. Following usual algebraic conventions, we abbreviate $x\cdot y$ by $xy$.

It is well-known (see, e.g, \cite{BC2014}) that item (3) in the definition of residuated lattices may be replaced by the following four equations:\footnote{We will use $\eq$ to denote formal equality.}
\begin{align*}
&x(y\join z) \eq xy\join xz,\\
&x\to (y\meet z) \eq (x\to y)\meet (x\to z),\\
&(x(x\to y))\join y \eq y,\\
&(x\to (xy))\meet y \eq y.
\end{align*}
The class of residuated lattices is therefore a variety (i.e., a class of algebras defined by equations, or equivalently, closed under taking homomorphic images, subalgebras, and direct products).

A residuated lattice is called a \emph{bounded commutative generalized basic logic algebra} if it satisfies the condition
$$(\forall x,y)(x\leq y \implies \exists z (yz\eq x)),$$
often referred to as \emph{divisibility}. One may show that this is equivalent to the identity
$$x(x\to y) \eq x\meet y.$$
Bounded commutative generalized basic logic algebras are sometimes called GBL$_{\sf ewf}$-algebras, where the subscripts refer to the {\bf e}xchange and {\bf w}eakening rules of structural proof theory, and the presence of a falsity constant $f$. These logical properties are algebraized by commutativity, integrality (i.e., the demand that the monoid identity is the greatest element), and boundedness. Since we always assume these properties in this note, we simply call bounded commutative generalized basic logic algebras \emph{GBL-algebras}.

GBL-algebras have been studied from both algebraic and logical viewpoints (see \cite{GJ2009}). Most relevantly, GBL-algebras are motivated as a generalization of BL-algebras (defined below), which provide an algebraic semantics for H\'ajek's basic fuzzy logic \cite{H1998}. Since they may be defined relative to residuated lattices by an equation, GBL-algebras form a variety. Importantly, GBL-algebras have distributive lattice reducts \cite[Lemma 2.9]{GT2005}.

This study focuses mostly on GBL-algebras, and will apply the following definitions only in that context. However, we state the definitions in more generality.
\begin{definition}
Let $\m A$ be a residuated lattice and $\sf V$ be a variety of residuated lattices.
\begin{enumerate}[\normalfont (1)]
\item Given a positive integer $k$, we say that $\m A$ is \emph{$k$-potent} if $\m A$ satisfies the identity $x^{k+1}\eq x^k$.
\item We say that $\m A$ is \emph{almost finite} if there exists a positive integer $k$ such that $\m A$ is $k$-potent.
\item We say that ${\sf V}$ has the \emph{finite model property} (or \emph{FMP}) if $\sf{V}$ is generated by its finite members.
\item We say that ${\sf V}$ has the \emph{almost finite model property} (or \emph{AFMP}) if $\sf V$ is generated by its almost finite members.
\end{enumerate}
\end{definition}
The following lemma explains the choice of the terminology ``almost finite.'' It is part of the folklore of residuated lattices, and we omit its easy proof (but caution that integrality is indispensable).
\begin{lemma}\label{lem:almost finite}
Let $\m A$ be a finite residuated lattice. Then there exists a positive integer $k$ such that $\m A$ satisfies $x^{k+1}\eq x^k$. Stated differently, every finite residuated lattice is almost finite. Moreover, every variety of residuated lattices with the FMP has the AFMP.
\end{lemma}

Note that there exist varieties of GBL-algebras with the AFMP but lacking the FMP.\footnote{In fact, there are uncountably-many undecidable superintuitionistic logics by a simple counting argument (see, e.g., \cite[Chapter 16]{CZ1997}). Each of these corresponds to a variety of Heyting algebras without the FMP, all of which are varieties of GBL-algebras with the AFMP.} On the other hand, the variety generated by Chang's MV-algebra is an example of a subvariety of GBL-algebras lacking the AFMP entirely.\footnote{This essentially follows from \cite[Proposition 5]{BDiNG2007}.}

Several subvarieties of GBL-algebras are especially notable:
\begin{itemize}
\item A \emph{Heyting algebra} is a $1$-potent \gbl-algebra.
\item A \emph{BL-algebra} is a \gbl-algebra that satisfies $(x\to y)\join (y\to x) \eq 1$.
\item A \emph{G\"odel algebra} is a Heyting algebra that satisfies $(x\to y)\join (y\to x) \eq 1$.
\item An \emph{MV-algebra} is a BL-algebra that satisfies $x \eq (x\to 0) \to 0$.
\item A \emph{Boolean algebra} is a Heyting algebra that $x \eq (x\to 0) \to 0$.
\end{itemize}
We denote the varieties of GBL-algebras, Heyting algebras, BL-algebras, G\"odel algebras, and MV-algebras by ${\sf GBL}$, $\sf HA$, ${\sf BL}$, $\sf GA$, and $\sf MV$, respectively. If $k$ is a positive integer and ${\sf V}$ is a variety of residuated lattices, we denote by ${\sf V}^k$ the subvariety of $k$-potent members of ${\sf V}$. If $\sf V$ is a subvariety of $\sf GBL$, then ${\sf V}^1$ is a variety of Heyting algebras; in particular, ${\sf GBL}^1 = {\sf HA}$, ${\sf BL}^1={\sf GA}$, and ${\sf MV}^1$ is the variety of Boolean algebras.

Note that each of ${\sf GBL}$, ${\sf HA}$, ${\sf GA}$, ${\sf BL}$, and ${\sf MV}$ has the FMP (the result is well-known for ${\sf HA}$; for ${\sf GBL}$, see \cite[Section 5]{JM2009}; for ${\sf BL}$, ${\sf MV}$, and ${\sf GA}$, see e.g. \cite{CHN2014}). Moreover, note that a subdirectly irreducible distributive residuated lattice satisfies $(x\to y)\join (y\to x)\eq 1$ if and only if it is totally-ordered \cite{BT2003}, whence each of $\sf MV$, $\sf BL$, and $\sf GA$ is generated by its totally-ordered members.

We will often refer to some special MV-algebras in the sequel. We denote by $\mv$ the standard MV-algebra chain, i.e., the MV-algebra $([0,1],\meet,\join,\cdot,\to,0,1)$ with operations defined by
\begin{align*}
x\meet y &= \min\{x,y\}\\
x\join y &= \max\{x,y\}\\
x\cdot y &= \max\{0, x+y-1\}\\
x\to y &= \min\{1,1-x+y)\},
\end{align*}
where $\max$, $\min$, $+$, and $-$ are as usual in the real unit interval $[0,1]$. For each positive integer $k\geq 2$, we denote by ${\textbf{\L}}_k$ the subalgebra of $\mv$ whose underlying universe is
$$\L_k = \Bigg\{0,\frac{1}{k-1},\dots,\frac{k-2}{k-1},1\Bigg\}.$$
Of course, ${\textbf{\L}}_2$ is isomorphic to the two-element Boolean algebra, and we sometimes denote it by $\m 2$ as is customary. We further adopt the convention that $\textbf{\L}_0=\mv$ and $\textbf{\L}_1$ is the one-element MV-algebra. Note that the set $\{{\textbf{\L}}_k : k\geq 2\}$  generates the variety $\sf MV$ (see \cite[Proposition 8.1.2]{CDM2000}).

The \emph{conuclear image} construction---which we now outline---is central to the theory of residuated lattices generally, and this work in particular. If $\m A$ is a residuated lattice, a \emph{conucleus} on ${\m A}$ is a function $\sigma\colon A\to A$ such that for all $x,y\in A$,
\begin{align*}
\sigma (x)&\leq x,\\
\sigma(\sigma (x)) &= \sigma(x),\\
x\leq y\implies& \sigma (x)\leq \sigma (y),\\
\sigma (x)\sigma (y)&\leq \sigma (xy),\\
\sigma(1)\sigma(x) = \sigma&(x)\sigma(1)=\sigma(x).
\end{align*}
If $\m A$ is a residuated lattice and $\sigma$ is a conucleus on ${\m A}$, then the $\sigma$-image of ${\m A}$,
$$A_\sigma = \sigma[A] = \{x\in A : \sigma(x)=x\},$$
may be endowed with the structure of a residuated lattice by defining $\meet_\sigma$ and $\to_\sigma$ on $A_\sigma$ by
$$x\meet_\sigma y = \sigma (x\meet y),$$
$$x\to_\sigma y = \sigma (x\to y).$$
We denote the resulting residuated lattice $(A_\sigma,\meet_\sigma,\join,\cdot,\to_\sigma,0,\sigma(1))$ by ${\m A}_\sigma$. 

\subsection{Filters and values.} If $\m A$ is a residuated lattice, a \emph{deductive filter} of $\m A$ is a nonempty subset $F\subseteq A$ such that:
\begin{enumerate}[\normalfont (1)]
\item $F$ is an up-set (i.e., $x\leq y$ and $x\in F$ implies $y\in F$).
\item If $x,y\in F$, then $xy\in F$.
\end{enumerate}
In particular, since $xy\leq x\meet y$ holds in every residuated lattice, the deductive filters of $\m A$ are lattice filters of (the lattice reduct of) ${\m A}$. The deductive filters of $\m A$ correspond bijectively to the congruences of $\m A$ by the map
$$F\mapsto \{(x,y)\in A\times A : (x\to y)\meet (y\to x)\in F\}.$$
If $F$ is a deductive filter of ${\m A}$, we denote the quotient of ${\m A}$ by the corresponding congruence by ${\m A}/F$.

An algebra is called \emph{subdirectly irreducible} if and only if it has a least congruence above the diagonal congruence. The above correspondence entails that a residuated lattice ${\m A}$ is subdirectly irreducible if and only if it has a least nontrivial deductive filter. If additionally ${\m A}$ is almost finite, then ${\m A}$ is subdirectly irreducible if and only if there is a greatest element of ${\m A}$ strictly smaller than $1$ by \cite[p. 202]{GJKO2007}.

A deductive filter $F$ of $\m A$ is called \emph{prime} if whenever $x,y\in A$ and $x\join y\in F$, we have $x\in F$ or $y\in F$. A \emph{value} of $\m A$ is a deductive filter $F$ so that there exists $x\in A$ such that $F$ is maximal among the deductive filters not containing $x$. Values are prime filters, and the quotient of a residuated lattice by one of its values is subdirectly irreducible (see, e.g., \cite[Lemma 5.26 and 5.27]{GJKO2007}). We will denote the set of values of a residuated lattice ${\m A}$ by $\Delta({\m A})$.

\subsection{Algebraic substructural logic.}\label{sec:algebraic substructural logic}
The idea animating algebraic logic is to view \emph{formulas} of a logic as \emph{terms} in a corresponding algebraic language. In well-behaved cases, theoremhood in the logic can be described in reference to the validity of certain equations involving those terms, and logical consequence becomes a matter of \emph{equational} consequence. Our interest in residuated algebraic structures arises because they provide the equivalent algebraic semantics \cite{BP1989} of substructural logics, in particular axiomatic extensions of the Full Lambek calculus (see, e.g., \cite{GO2006,GJKO2007}). Residuated lattices (in the sense of Section \ref{sec:residuated lattices}) give the equivalent algebraic semantics of the Full Lambek calculus with exchange, weakening, and extended by the axiom $0\vdash \varphi$. Moreover, the lattice of axiomatic extensions of the aforementioned logic is dually isomorphic to the lattice of subvarieties of residuated lattices. In particular, {\sf BL} algebraizes H\'ajek's basic logic, {\sf MV} algebraizes \L ukasiewicz logic, {\sf HA} algebraizes intuitionistic logic, and $\sf GA$ algebraizes G\"odel-Dummett logic. The logic corresponding to {\sf GBL} is most often discussed algebraically, but {\sf GBL} also provides an algebraic semantics for {\bf I\L L} \cite{LOR2020}.

In order to express the needed aspects of this algebra-logic correspondence, we fix a set of propositional variable symbols $\var$. Like every algebraic language, the language $\mathcal{L}=\{\meet,\join,\cdot,\to,0,1\}$ of residuated lattices has a term algebra over the variables ${\sf Var}$, which we denote by ${\bf Tm}$. The algebra ${\bf Tm}$ is freely-generated by the set $\var$. If ${\m A}$ is a residuated lattice, an \emph{algebraic assignment in ${\m A}$} is a function $h\colon {\sf Var}\to {\m A}$. By the free property of the term algebra, every algebraic assignment $h$ in ${\m A}$ uniquely extends to a homomorphism $\hat{h}\colon{\bf Tm}\to{\m A}$, which may be defined recursively by setting $\hat{h}(p)=h(p)$ for all $p\in\var$, $\hat{h}(0)=0$ and $\hat{h}(1)=1$, and $\hat{h}(t_1\star t_2) = \hat{h}(t_1)\star\hat{h}(t_2)$ for $\star\in\{\meet,\join,\cdot,\to\}$ (see, e.g., \cite[Chapter II, \S 10]{BS1981} for a standard treatment).

If $\varphi,\psi\in {\bf Tm}$, the inequality $\varphi\leq\psi$ is \emph{valid} in ${\m A}$ if $\hat{h}(\varphi)\leq \hat{h}(\psi)$ for every algebraic assignment $h\colon {\sf Var}\to {\m A}$. Now suppose $L$ is an extention of the Full Lambek calculus that is algebraized by a variety $\sf V$ of residuated lattices. Viewing $\psi_1,\ldots,\psi_n,\varphi\in Tm$ as formulas of $L$, the sequent $\psi_1,\ldots,\psi_n\vdash \varphi$ is provable in $L$  if and only if for each ${\m A}\in {\sf V}$ and each algebraic assignment $h$ into ${\m A}$,
$$\hat{h}(\psi_1)=1\;\&\; \hat{h}(\psi_2)=1\;\&\;\ldots\;\&\; \hat{h}(\psi_n)=1 \Rightarrow \hat{h}(\varphi)=1.$$
Thanks to the local deduction theorem \cite{GO2006}, the sequent $\psi_1,\ldots,\psi_n\vdash \varphi$ is provable in $L$ if and only if there exists a non-negative integer $k$ such that $\psi_1^k\cdot\psi_2^k\ldots\cdot\psi_n^k\to \varphi \eq 1$ is valid in every ${\m A}\in {\sf V}$.\footnote{A consequence of this is that the validity of equations in the algebraic semantics determine the validity of sequents in the logic.}

In the sequel, our primary focus is on logics algebraized by subvarieties of $\sf GBL$. This family of logics was originally introduced via their algebraic counterparts, viz. GBL-algebras, and the latter remain the most natural environment for their study. Due to this fact, as well as the close relationship between substructural logics and their equivalent algebraic semantics, we will not belabor the distinction between axiomatic extensions of the Full Lambek calculus and the varieties algebraizing them. We will adopt an algebraic mindset here, and will work primarily with varieties of residuated lattices rather than directly with the logics they algebraize.

\section{Poset products}\label{sec:poset products}
\subsection{The poset product construction.} Poset products of residuated lattices were introduced in \cite{JM2009} under the name \emph{poset sums}, and were further developed in \cite{J2009,JM2010,BG2018,BM2011} in order to give embedding and representation theorems for various classes of GBL-algebras. We recall pertinent details of the poset product construction. Our discussion essentially synthesizes information contained in \cite{J2009}, \cite{JM2009}, and \cite{JM2010}, and specializes this to the case of bounded, commutative, and integral algebras.

Suppose that $(X,\leq)$ is a poset and $\{{\m A}_x : x\in X\}$ is an indexed collection of residuated lattices sharing a common least element $0$ and common greatest element $1$. We set
$${\m B}=\prod_{x\in X} {\m A}_x$$
and define a map $\Box\colon B\to B$ by
\[ \Box(f)(x) = \begin{cases}
      f(x) & \text{ if }f(y)=1\text{ for all }y>x \\
      0 & \text{ if there exists }y>x\text{ with } f(y)\neq 1.
   \end{cases}
\]
The map $\Box$ is a conucleus on ${\m B}$ by \cite[Lemma 9.4]{JM2010}.\footnote{As recalled in \cite{JM2010}, every Heyting algebra is the image of a Boolean algebra under a conucleus, i.e., the modal necessity operator of the corresponding S4 modal algebra. We use the symbol $\Box$ because our conucleus generalizes this situation.} The \emph{poset product} of the indexed family $\{{\m A}_x : x\in X\}$ is the algebra ${\m B}_\Box$, which we will sometimes denote
$$\prod_{(X,\leq)} {\m A}_x.$$
Observe that we may distinguish a poset product from a direct product by the index, which is the poset $(X,\leq)$ rather than the set $X$. If $(X,\leq)$ is a poset and $\{{\m A}_x : x\in X\}$ is an indexed family of residuated lattices, the \emph{dual poset product} of the indexed family is the poset product
$$\prod_{(X,\geq)} {\m A}_x$$
with respect to the order dual $(X,\geq)$ of the poset $(X,\leq)$.

Note that in the poset product ${\m B}_{\Box}$, meets, joins, and products are computed pointwise (see, e.g., \cite[Lemma 2]{J2009}). Since $\Box 0 = 0$ and $\Box 1 =1$ as well, we obtain the following.

\begin{lemma}\label{lem:product subalgebra}
The poset product ${\m B}_\Box$ is a $\{\meet,\join,\cdot,0,1\}$-subalgebra of the direct product ${\m B}$, and hence satisfies every equation in the language $\{\meet,\join,\cdot,0,1\}$ that is satisfied in ${\m B}$.
\end{lemma}

The following fact is useful for working with embeddings into poset products.

\begin{lemma}\label{lem:subalgebras of poset products}
Let $(X,\leq)$ be a poset, and let ${\m A}_x$, ${\m B}_x$ be residuated lattices for each $x\in X$. Suppose further that ${\m A}_x$ is a subalgebra of ${\m B}_x$ for each $x\in X$. Then $\prod_{(X,\leq)} {\m A}_x$ is a subalgebra $\prod_{(X,\leq)} {\m B}_x$.
\end{lemma}

\begin{proof}
Note that ${\m A}=\prod_{x\in X} {\m A}_x$ is a subalgebra of ${\m B}=\prod_{x\in X} {\m B}_x$ since ${\m A}_x$ is a subalgebra of ${\m B}_x$ for each $x\in X$. Moreover, $\Box$ is a conucleus on each of ${\m A}$ and ${\m B}$, and it is obvious that the operations of ${\m A}_\Box$ are the restrictions of those on ${\m B}_\Box$. The result follows.
\end{proof}

The next lemma provides several alternative descriptions of the elements of ${\m B}_\Box$.

\begin{lemma}\label{lem:aclabeling}
Let $f\in {\m B}$. The following are equivalent.
\begin{enumerate}[\normalfont (1)]
\item $f\in {\m B}_\Box$.
\item $\Box f = f$.
\item For all $x,y\in X$ with $x<y$, $f(x)=0$ or $f(y)=1$.
\item $S_f=\{x\in X : f(x)\notin\{0,1\}\}$ is a (possibly empty) antichain of $(X,\leq)$, $L_f=f^{-1}(0)$ is a down-set of $(X,\leq)$, and $U_f=f^{-1}(1)$ is an up-set of $(X,\leq)$.
\end{enumerate}
\end{lemma}

\begin{proof}
(1) and (2) are equivalent by the general theory of conuclei. (3) and (4) are equivalent by \cite[Lemma 1]{J2009}. The rest follows from \cite[Lemma 9.4]{JM2010}.
\end{proof}
Owing to item (4), the choice functions satisfying the equivalent conditions of Lemma \ref{lem:aclabeling} are called \emph{antichain labelings}, or \emph{ac-labelings} for short.

The subsets $S_f$, $L_f$, and $U_f$ identified in item (4) are quite useful for thinking about poset products. Observe that for any ac-labeling $f$, the three sets $L_f$, $S_f$, and $U_f$ form a partition of the index set $X$.

\begin{lemma}\label{lem:comparability}
Let $f,g\in {\m B}_{\Box}$ be ac-labelings. The following are equivalent.
\begin{enumerate}[\normalfont (1)]
\item $f\leq g$.
\item $L_g\subseteq L_f$, $U_f\subseteq U_g$, and $f(x)\leq g(x)$ for all $x\in S_f\cap S_g$.
\end{enumerate}
\end{lemma}

\begin{proof}
Assume (1) and let $x\in X$. Note that if $g(x) = 0$ then $f\leq g$ implies $f(x)=0$, giving $L_g\subseteq L_f$. On the other hand, if $f(x)=1$, then $g(x)=1$ and we get $U_f\subseteq U_g$. That $f(x)\leq g(x)$ for all $x\in S_f\cap S_g$ is immediate from the hypothesis, so (2) follows.

Now for the converse assume (2). If $g(x)=0$, then $f(x)=0$ since $L_g\subseteq L_f$, so $f(x)\leq g(x)$. If $g(x)=1$, then $f(x)\leq g(x)$ follows automatically. In the only remaining case, $x\in S_g$. If $x\in S_f$ as well, then $x\in S_f\cap S_g$ and $f(x)\leq g(x)$ by hypothesis. If $x\notin S_f$, then $x\in L_f$ or $x\in U_f$. If $x\in U_f$, then $U_f\subseteq U_g$ implies $g(x)=1$, contradicting $x\in S_g$. On the other hand, if $x\in L_f$ then $f(x)=0\leq g(x)$. In every case, we get $f(x)\leq g(x)$ and the result follows.
\end{proof}

\subsection{Closure under poset products and embedding theorems.}
For our intended application of poset products to relational semantics, two questions concern us:
\begin{enumerate}[\normalfont (Q1)]
\item If $\sf V$ is a variety of residuated lattices, what conditions on the poset $(X,\leq)$ and the residuated lattices ${\m A}_x$, $x\in X$, guarantee that the corresponding poset product $\prod_{(X,\leq)} {\m A}_x$ is in ${\sf V}$?
\item If $\m A$ is a residuated lattice, what hypotheses on $(X,\leq)$ and ${\m A}_x$, $x\in X$, guarantee that we may embed ${\m A}$ in $\prod_{(X,\leq)} {\m A}_x$?
\end{enumerate}
The first question turns out to be closely related to the soundness of the corresponding relational semantics, whereas the second is closely related to completeness (cf. Lemma \ref{thm:general condition}).

Recall that a poset $(X,\leq)$ is a \emph{root system} if $\up x = \{y\in X : x\leq y\}$ is totally-ordered for each $x\in X$. The following lemma provides a partial answer to (Q1); it is mostly a summary of what is known in the literature regarding this question.
\begin{lemma}\label{lem:closure under poset prod}
Let $(X,\leq)$ be a poset and let $\{{\m A}_x : x\in X\}$ be an indexed family of residuated lattices. Set
$${\m B} = \prod_{(X,\leq)} {\m A}_x.$$
Then:
\begin{enumerate}[\normalfont (1)]
\item If ${\m A}_x$ is a GBL-algebra for each $x\in X$, then ${\m B}$ is a GBL-algebra.
\item If ${\m A}_x$ is an MV-algebra for each $x\in X$, then ${\m B}$ is a GBL-algebra.
\item If ${\m A}_x$ is $k$-potent for each $x\in X$, then ${\m B}$ is $k$-potent.
\item If ${\m A}_x$ is $\m 2$ for each $x\in X$, then ${\m B}$ is a Heyting algebra.
\item If $(X,\leq)$ is a root system and ${\m A}_x$ is a totally-ordered MV-algebra for each $x\in X$, then ${\m B}$ is a BL-algebra.
\item If $(X,\leq)$ is a root system and ${\m A}_x$ is $\m 2$ for each $x\in X$, then ${\m B}$ is a G\"odel algebra.
\item If $(X,\leq)$ is a chain and ${\m A}_x$ is a chain for each $x\in X$, then ${\m B}$ is also a chain.
\item If $(X,\leq)$ is a chain and ${\m A}_x$ is an MV-algebra chain for each $x\in X$, then ${\m B}$ is a totally-ordered BL-algebra.
\item If $(X,\leq)$ is a chain and ${\m A}_x$ is $\m 2$ for each $x\in X$, then ${\m B}$ is a totally-ordered G\"odel algebra.
\end{enumerate}
\end{lemma}

\begin{proof}
(1) immediate from \cite[Theorem 29, Theorem 30]{JM2009}. (2) is immediate from (1) since MV-algebras are, in particular, GBL-algebras. Lemma~\ref{lem:product subalgebra} implies (3), and (4), (5), and (6) follow from \cite[Theorem 6.2]{JM2010}.

For (7), let $f,g\in\m B$. The up-sets of a chain are totally-ordered under inclusion, so we have $U_f\subseteq U_g$ or $U_g\subseteq U_f$; assume $U_f\subseteq U_g$ without loss of generality. Then we have $U_g^\mathsf{c}\subseteq U_f^\mathsf{c}$ by taking complements. Since $(X,\leq)$ is a chain and $S_f$, $S_g$ are antichains of $(X,\leq)$, the cardinalities of each of $S_f$ and $S_g$ are at most $1$. If $S_f=\emptyset$, then we have that $L_g\subseteq U_g^\mathsf{c}\subseteq U_f^\mathsf{c}=L_f$. Since $f(x)\leq g(x)$ for all $x\in S_f\cap S_g$ vacuously, it follows from Lemma~\ref{lem:comparability} that $f\leq g$ in this case. For the only remaining case, assume $S_f=\{x\}$ is a singleton. Then $U_f = \{y\in X : x < y\}$ and $L_f=\{y\in X : y < x\}$. If $x\in S_g$ as well, then $U_g=U_f$ and $L_g=L_f$, and Lemma~\ref{lem:comparability} gives $f\leq g$ or $g\leq f$ according to whether $f(x)\leq g(x)$ or $g(x)\leq f(x)$ (and one of these must hold, since ${\m A}_x$ is a chain). If $x\notin S_g$, then either $x\in U_g$ (in which case $f\leq g$) or $x\in L_g$ (in which case $g\leq f$). This shows ${\m B}$ is a chain.

Since chains are in particular root systems, (8) and (9) are immediate consequences of (5) and (6), respectively.
\end{proof}

We now turn to (Q2). In \cite{JM2010}, Jipsen and Montagna provided embedding theorems for various classes of GBL-algebras in a quite general setting (not necessarily integral, commutative, or bounded). We will not need the full strength of these results for our present purposes, and we confine our discussion to almost finite GBL-algebras. The embedding theorems are simpler in this setting, and we will briefly sketch them. The following material is drawn from \cite{JM2009} and \cite{JM2010}.

Let ${\m A}$ be a $k$-potent GBL-algebra and let $\Delta = \Delta({\m A})$ be its set of values. For each $x\in \Delta$, the quotient algebra ${\m A}/x$ is subdirectly irreducible and hence has a minimum nontrivial deductive filter. This minimum nontrivial deductive filter comprises a $\{\meet,\join,\cdot,\to,1\}$-subalgebra of ${\m A}/x$. One may show that this subalgebra is the $0$-free reduct of a finite simple $k$-potent MV-algebra chain ${\m A}_x$ (i.e., ${\m A}_x$ is isomorphic to some ${\textbf{\L}}_m$, $2\leq m\leq k+1$, by \cite[Corollary 3.5.4]{CDM2000}). The poset product of the resulting family will be used in the sequel, so we introduce the following notation:
$$E({\m A})=\prod_{(\Delta,\subseteq)} {\m A}_x.$$
Further, for each $a\in A$, we define a choice function $\varepsilon_a\in\prod_{x\in\Delta} {\m A}_x$ by
\[ \varepsilon_a(x) = \begin{cases}
      a/x & \text{ if }a/x\in A_x \\
      0 & \text{ otherwise.}
   \end{cases}
\]
With the above set-up, we obtain the following.
\begin{lemma}[{\cite[Theorem 4.1]{JM2010}}]\label{lem:embed almost finite}
Let ${\m A}$ be a $k$-potent GBL-algebra. Then each $\varepsilon_a$ is an ac-labeling and ${\m A}$ embeds in the poset product $E({\m A})$ via the map $a\mapsto\varepsilon_a$.
\end{lemma}

\begin{remark}
We have presented the foregoing construction somewhat differently than \cite{JM2010}. First, that paper states the embedding result only for the $0$-free signature, but inspection of the proof confirms that the embedding also preserves $0$. Second, in \cite{JM2010} the subdirectly irreducible GBL-algebra ${\m A}/x$ is decomposed as an ordinal sum ${\m B}_x\oplus {\m W}_x$ of a proper subalgebra ${\m B}_x$ of ${\m A}/x$ and (the 0-free reduct of) a nontrivial $k$-potent MV-algebra chain ${\m W}_x$ with at most $k+1$ elements. Note that in this context \emph{ordinal sum} refers to the notion defined for (integral) GBL-algebras in \cite{JM2009} rather than the more widely-known notion for hoops (as depicted in, e.g., \cite{BF2000}). The algebras ${\m W}_x$ are used as the factors of the poset product in \cite{JM2010}. Clearly ${\m W}_x$ is a nontrivial deductive filter. Since nontrivial finite $k$-potent MV-algebra chains are simple algebras, it follows that ${\m W}_x$ is the least nontrivial deductive filter of ${\m A}/x$, i.e., ${\m W}_x={\m A}_x$. We refer to the reader to \cite{JM2009} for a full discussion of ordinal sums in this context. Compare also with \cite[Proof of Theorem 16]{J2009}.
\end{remark}

For finite GBL-algebras, the latter embedding result may be upgraded to a \emph{representation}. In the sequel, we will apply the following lemma to study subvarieties of ${\sf GBL}$ with the FMP.

\begin{lemma}[{\cite[Theorem 33]{JM2009}}]\label{lem:finite rep}
Let ${\m A}$ be a finite GBL-algebra. Then there exist finite simple MV-algebras ${\m A}_x$, $x\in\Delta=\Delta({\m A})$, such that
$${\m A}\cong\prod_{(\Delta,\subseteq)} {\m A}_x.$$
Moreover, ${\m A}_x$ is a $\{\meet,\join,\cdot,\to,1\}$-subalgebra of a quotient of ${\m A}$ for each $x\in\Delta$, whence each ${\m A}_x$ satisfies all $0$-free identities that are satisfied in ${\m A}$.
\end{lemma}

\section{Frames, relational models, and satisfaction}\label{sec:frames}
We now arrive at our main topic of interest: Extracting relational semantics from poset products.
\subsection{Frames and satisfaction.}\label{sec:frames and satisfaction}
\begin{definition}\label{def:frame}
A \emph{frame} is an ordered triple $(X,\leq,\mathbb{A})$, where
\begin{enumerate}[\normalfont (1)]
\item $(X,\leq)$ is a poset.
\item $\mathbb{A}=\{{\m A}_x : x\in A\}$ is an indexed family of residuated lattices.
\end{enumerate}
If $\mathsf{K}$ is a class of posets, we say that the frame $(X,\leq,\mathbb{A})$ is \emph{$\mathsf{K}$-based} or \emph{based in $\mathsf{K}$} when $(X,\leq)\in\mathsf{K}$. Likewise, if $\mathsf{V}$ is a class of residuated lattices, we say that $(X,\leq,\mathbb{A})$ is \emph{$\mathsf{V}$-valued} or \emph{valued in $\mathsf{V}$} when ${\m A}_x\in\mathsf{V}$ for every $x\in X$.

For any frame $F=(X,\leq,\{{\m A}_x : x\in X\})$, we define
$$P(F) = \prod_{(X,\leq)} {\m A}_x$$
to be the \emph{poset product associated to $F$}. A \emph{valuation} in a frame $(X,\leq, \{{\m A}_x\}_{x\in X})$ is a function $h\colon\var\to P(F)$. A \emph{relational model} is an ordered pair $(F,h)$, where $F$ is a frame and $h$ is a valuation in $F$. We say that a model $(F,h)$ is $\mathsf{K}$-based, based in $\mathsf{K}$, $\mathsf{K}$-valued, or valued in $\mathsf{K}$ whenever its underlying frame $F$ is. When $\sf K$ is a singleton, we drop the exterior braces (and thus refer to ${\m A}$-valued models, and so forth).
\end{definition}

If $(F,h)$ is a relational model, then $h$ is in particular an algebraic assignment into the poset product $P(F)$. We define for each $x\in X$,
$$(F,h),x\km\varphi \iff \hat{h}(\varphi)(x)=1.$$
Following the usual terminology of Kripke semantics, we say \emph{$\varphi$ is true at $x$ in $(F,h)$} when the above holds. We obtain a definition of $\km$ that is more familiar in relational semantics by explicitly writing $\hat{h}$ in terms of its recursive definition. In particular, $\hat{h}$ may be defined recursively by setting $\hat{h}(0)=0,$ $\hat{h}(1)=1,$ $\hat{h}(p)(x) = h(p)(x)$ for each $p\in\var$, and extending this definition to terms by
$$\hat{h}(\varphi\meet\psi)(x) = \hat{h}(\varphi)(x)\meet \hat{h}(\psi)(x),$$
$$\hat{h}(\varphi\join\psi)(x) = \hat{h}(\varphi)(x)\join \hat{h}(\psi)(x),$$
$$\hat{h}(\varphi\cdot\psi)(x) = \hat{h}(\varphi)(x)\cdot \hat{h}(\psi)(x),$$
$$\hat{h}(\varphi\to\psi)(x) = \Box (\hat{h}(\varphi)\to \hat{h}(\psi))(x),$$

We further define $(F,h)\km\varphi$ provided that $(F,h),x\km\varphi$ for all $x\in X$, and in this case we say \emph{$\varphi$ is true in $(F,h)$}. For a frame $F$, we define $F\km\varphi$ if $(F,h)\km\varphi$ for every valuation $h$ into $F$, and say that \emph{$\varphi$ is valid in $F$}. If $\sf K$ is a class of frames, we further say that $\varphi$ is \emph{valid in $\sf K$} and write $\sf K\km\varphi$ provided that $\varphi$ is valid in $F$ for each $F\in\sf K$.
\begin{lemma}\label{lem:forcing and algebra}
Let $F=(X,\leq,\{{\m A}_x : x\in X\})$ be a frame and $\varphi\in{\bf Tm}$. The following are equivalent.
\begin{enumerate}[\normalfont (1)]
\item $F\km\varphi$.
\item $\varphi\eq 1$ is valid in $P(F)$.
\end{enumerate}
\end{lemma}

\begin{proof}
We have:
\begin{align*}
F\km\varphi &\iff (F,h)\km\varphi\text{ for all }h\colon\var\to P(F)\\
&\iff (F,h),x\km\varphi\text{ for all }h\colon\var\to P(F),x\in X\\
&\iff \hat{h}(\varphi)(x)=1\text{ for all }h\colon\var\to P(F),x\in X\\
&\iff \hat{h}(\varphi)\equiv 1\text{ for all }h\colon\var\to P(F)\\
&\iff \varphi\eq 1 \text{ is valid in }P(F).\\
\end{align*}
This yields the claim.
\end{proof}

\subsection{Comparison to the literature.}\label{sec:comparison}

The semantics introduced in Section \ref{sec:frames and satisfaction} unifies and generalizes several kinds of relational semantics appearing in the literature. We will now make these connections explicit.\\

\noindent {\bf Intuitionistic Kripke frames.} The celebrated Kripke semantics for intuitionistic logic is the most widely-known relational semantics for any substructural logic, and is treated extensively in, e.g., \cite[Chapter 2]{CZ1997}. In this setting, the term `frame' is used as a synonym for `poset,' and a `valuation' consists of a map from the set of propositional variables into the up-sets of that poset. The following connects this usage to our terminology.
\begin{proposition}\label{prop:Kripke connection}
Let $(X,\leq,\{{\m A}_x : x\in X\})$ be a frame valued in ${\m 2}$, and let $h\colon\var\to\prod_{x\in X} {\m A}_x$. The following are equivalent.
\begin{enumerate}[\normalfont (1)]
\item $h$ is a valuation in $(X,\leq,\{{\m A}_x : x\in X\})$.
\item $\{x\in X : h(p)(x) = 1\}$ is a up-set of $(X,\leq)$ for each $p\in\var$.
\end{enumerate}
\end{proposition}

\begin{proof}
If $h$ is a valuation in $(X,\leq,\{{\m A}_x : x\in X\})$, then $h(p)\in\prod_{(X,\leq)} {\m A}_x$ for each $p\in\var$ and by Lemma~\ref{lem:aclabeling}(4) we have that $U_{h(p)}=\{x\in X : h(p)(x)=1\}$ is an up-set of $(X,\leq)$. This gives (1) implies (2).

For the converse, suppose that $\{x\in X : h(p)(x) = 1\}$ is a up-set of $(X,\leq)$ for each $p\in\var$. It suffices to show that each function $h(p)$ is an ac-labeling, so suppose that $x<y$ in $X$. If $h(p)(x)\neq 0$, then since $h(p)$ only takes values in $\{0,1\}$ we have $h(p)(x)=1$. Hence by hypothesis $h(p)(y)=1$. It follows that $h(p)(x)=0$ or $h(p)(y)=1$, giving the result. 
\end{proof}
Proposition~\ref{prop:Kripke connection} shows that Definition~\ref{def:frame} agrees with the usual intuitionistic definition: An intuitionistic frame is a $\m 2$-valued frame in our terminology, and if $h$ is a valuation then the corresponding intuitionistic valuation is the map from $\var$ into up-sets given by $p\mapsto U_{h(p)}$. It is straightforward to see that our definition of $\km$ also coincides with the usual intuitionistic forcing relation when we restrict our attention to $\bf 2$-valued frames.\\

\noindent {\bf Bova-Montagna structures.} In \cite{LOR2020}, Lewis-Smith, Oliva, and Robinson give a relational semantics for the logic \textbf{I\L L} in terms of \emph{Bova-Montagna structures}. The latter are exactly $\mv$-valued models in the sense of Definition~\ref{def:frame} (compare with \cite[Definition 3.3]{LOR2020}). Valuations are internalized in the forcing relation in that treatment, so frames (as opposed to models) are not considered.

Our definition of the forcing relation coincides with the one Lewis-Smith, Oliva, and Robinson provide (cf. \cite[Definition 3.6]{LOR2020}). However, there are several notational differences. First, antichain labelings are called \emph{sloping functions} in \cite{LOR2020}, and that paper denotes our $\Box$ by $\floor{{\bf inf}}_{v\succeq w}$. It makes no mention of the connection to conuclei or modal logic. Second, \cite{LOR2020} uses additive notation for MV-algebras instead of the multiplicative notation we adopt. The additive and multiplicative signatures of MV-algebras are term-equivalent, so this makes no substantive difference. However, the multiplicative signature is indispensable for generalization beyond frames valued in MV-algebras.\\

\noindent {\bf Temporal flows.} In \cite{ABM2009}, Aguzzoli, Bianchi, and Marra provide a semantics for H\'ajek's basic logic using \emph{temporal flows}. A temporal flow is a structure $(T,\leq,L)$, where $(T,\leq)$ is a poset and and $L\colon T\to\mathbb{N}$ is a function. A \emph{temporal assignment} is a map $v\colon\var\times T\to [0,1]$ such that for any $t,t'\in T$ and $p\in\var$:
\begin{enumerate}[\normalfont (1)]
\item $v(p,t)\in\textbf{\L}_{L(t)}$.
\item If $t\leq t'$, then $v(p,t)\leq v(p,t')$.
\item If $t\neq t'$ and $v(p,t),v(p,t')\in (0,1)$, then $t$ and $t'$ are incomparable.
\end{enumerate}
Clearly, each temporal flow $(T,\leq,L)$ induces a frame $(T,\leq,\mathbb{L})$ by setting $\mathbb{L}=\{\textbf{\L}_{L(t)} : t\in T\}$. Conversely, if $(X,\leq,\mathbb{A})$ is a frame valued in $\{\textbf{\L}_k : k\in\mathbb{N}\}$, then we may define a function $L\colon X\to\mathbb{N}$ by setting $L(x)=k$, where ${\m A}_x = \textbf{\L}_k$. The resulting structure $(X,\leq,L)$ is a temporal flow.
\begin{proposition}\label{prop:temporal assign and val}
Let $(T,\leq,L)$ be a temporal flow and $(X,\leq,\mathbb{A})$ be a frame valued in $\{\textbf{\L}_k : k\in\mathbb{N}\}$.
\begin{enumerate}[\normalfont (1)]
\item If $v\colon\var\times T\to [0,1]$ is a temporal assignment, then the map $\bar{v}\colon\var\to\prod_{t\in T} \textbf{\L}_{L(t)}$ given by $\bar{v}(p)(t)=v(p,t)$ is a valuation in the frame $(T,\leq,\mathbb{L})$, where $\{\textbf{\L}_{L(t)} : t\in T\}$.
\item If $h$ is a valuation in $(X,\leq,\{{\m A}_x : x\in X\})$, then the map $\bar{h}\colon\var\times X\to [0,1]$ defined by $\bar{h}(p,x)=h(p)(x)$ is a temporal assignment.
\end{enumerate}
\end{proposition}

\begin{proof}
For (1), let $p\in\var$. Note that property (1) in the definition of temporal assignments gives that $\bar{v}(p)\in\prod_{t\in T} \textbf{\L}_{L(t)}$, so that the definition of $\bar{v}$ makes sense. It is enough to show that $\bar{v}(p)$ is an ac-labeling. Observe that item (3) in the definition of temporal assignments shows that $\{t\in T : \bar{v}(t)\notin\{0,1\}\}$ is an antichain. On the other hand, item (2) entails that the $\bar{v}^{-1}(0)$ is a down-set and $\bar{v}^{-1}(1)$ is an up-set. The result then follows from Lemma~\ref{lem:aclabeling}(4).

For (2), we verify the three defining conditions for being a temporal assignment. Note that condition (1) is immediate from the fact that $h(p)\in\prod_{x\in X} {\m A}_x$ and ${\m A}_x$ is $\textbf{\L}_{L(x)}$ by definition. For item (2), note that if $x,y\in X$ with $x<y$, then either $\bar{h}(p,x)=0$ or $\bar{h}(p,y)=1$. We have $\bar{h}(p,x)\leq\bar{h}(p,y)$ in either of these cases, and (2) follows because $\bar{h}(p,x)\leq\bar{h}(p,y)$ also obviously holds if $x=y$. Item (3) is immediate from Lemma~\ref{lem:aclabeling}(4) since $\{x\in X : h(p)(x)\not\in\{0,1\}\}$ is an antichain, so we obtain the result.
\end{proof}
In light of the foregoing remarks, temporal flows are tantamount to $\{\textbf{\L}_k : k\in\mathbb{N}\}$-valued frames and temporal assignments are tantamount to valuations in $\{\textbf{\L}_k : k\in\mathbb{N}\}$-valued frames.

In H\'ajek's basic logic, all of the connectives are definable in terms of $\cdot$, $\to$, and $0$. Consequently, \cite{ABM2009} only extends temporal assignments to formulas in this signature. This is done by setting $v(0,t)=0$, $v(\varphi\cdot\psi,t)=v(\varphi,t)\cdot v(\psi,t)$, and 
\[ v(\varphi\to\psi,t) = \begin{cases}
      1 & \text{if } v(\varphi,t')\leq v(\psi,t')\text{ for all }t'\geq t,\\
      v(\varphi,t)\to v(\psi,t) & \text{if }v(\psi,t)<v(\varphi,t)<1\text{ and }\\
                                                 & v(\psi,t')=1\text{ for all } t'>t,\\
      v(\psi,t) & \text{otherwise,}
   \end{cases}
\]
where on the right $\to$ denotes the implication in $\mv$. In \cite{ABM2009}, a formula $\varphi$ is said to be valid in a temporal flow $(T,\leq,L)$ provided that $v(\varphi,t)=1$ for all temporal assignments $v$ and all $t\in T$. The next proposition shows that this notion of validity agrees with ours.

\begin{proposition}
Let $(X,\leq,\{{\m A}_x : x\in X\})$ be a $\{\textbf{\L}_k : k\in\mathbb{N}\}$-valued frame, let $h$ be an assignment in this frame. Further let $(X,\leq,L)$ be the corresponding temporal flow and $\bar{h}\colon\var\times X\to [0,1]$ be the corresponding temporal assignment as above. Let $\varphi$ be a formula (i.e., term) in the language $\{\cdot,\to,0\}$. Then:
\begin{enumerate}[\normalfont (1)]
\item $\bar{h}(\varphi,x)=\hat{h}(\varphi)(x)$ for all $x\in X$, i.e., the extension of $h$ as a temporal assignment agrees with its extension as a valuation.
\item $\varphi$ is valid in the frame $(X,\leq,\{{\m A}_x : x\in X\})$ if and only if $\varphi$ is valid in the temporal flow $(X,\leq,L)$.
\end{enumerate}
\end{proposition}

\begin{proof}
The proof of (1) is a straightforward induction on the complexity of $\varphi$. The base case is trivial, so suppose that for all formulas $\psi$ with complexity strictly less than that of $\varphi$, $\hat{h}(\psi)(x)=\bar{h}(\psi,x)$ for all $x\in X$. If $\varphi$ is $\psi_1\cdot\psi_2$, then $\hat{h}(\varphi)(x)=\bar{h}(\varphi,x)$ is immediate from the definitions. Now suppose that $\varphi$ is $\psi_1\to\psi_2$. We must show $\Box(\hat{h}(\psi_1)\to\hat{h}(\psi_2))(x)=\bar{h}(\psi_1\to\psi_2,x)$. The inductive hypothesis gives $\hat{h}(\psi_1)(x)=\bar{h}(\psi_1,x)$ and $\hat{h}(\psi_2)(x)=\bar{h}(\psi_2,x)$ for all $x\in X$. There are three cases, according to the clause used in the piecewise definition of the extension of $\to$ in temporal assignments.

First, if $\bar{h}(\psi_1,y)\leq \bar{h}(\psi_2,y)$ for all $y\geq x$, then $\bar{h}(\psi_1\to\psi_2,x)=1$. The inductive hypothesis gives $\hat{h}(\psi_1)(y)\leq \hat{h}(\psi_2)(y)$ for all $y\geq x$, and by residuation we have $(\hat{h}(\psi_1)\to\hat{h}(\psi_2))(y)=1$. The definition of $\Box$ then gives $\Box(\hat{h}(\psi_1)\to\hat{h}(\psi_2))(x)=\hat{h}(\psi_1)(x)\to\hat{h}(\psi_2)(x)=1$.

Second, if $\bar{h}(\psi_2,x)<\bar{h}(\psi_1,x)<1$ and $\bar{h}(\psi_2,y)=1$ for all $y>x$, then $\bar{h}(\psi_1\to\psi_2,x)=\bar{h}(\psi_1,x)\to\bar{h}(\psi_2,x)$. Moreover, for all $y>x$ we have $\hat{h}(\psi_1)(y)\to\hat{h}(\psi_2)(y)=\hat{h}(\psi_1)(y)\to 1 = 1$, so $\Box(\hat{h}(\psi_1)\to\hat{h}(\psi_2))(x)=\hat{h}(\psi_1)(x)\to\hat{h}(\psi_2)(x)=\bar{h}(\psi_1,x)\to\bar{h}(\psi_2,x)$.

Thirdly and finally, if neither of the previous cases hold there must exist $y\geq x$ such that $\bar{h}(\psi_1,y)\not\leq\bar{h}(\psi_2,y)$, i.e., $\bar{h}(\psi_2,y)<\bar{h}(\psi_1,y)$. We first consider the subcases where $x=y$. Then because we assume we are not in the second case, we have that $\bar{h}(\psi_1,x)=1$ or else there exists $z>x$ such that $\bar{h}(\psi_2,z)<1$. In the subcase that $\bar{h}(\psi_1,x)=1$, we have that $\hat{h}(\psi_1)(y)=1$ for all $y\geq x$ and $\hat{h}(\psi_1)(y)\to\hat{h}(\psi_2)(y)=\hat{h}(\psi_2)(y)$ for all $y\geq x$. Then $\Box(\hat{h}(\psi_1)\to\hat{h}(\psi_2))(x) = \hat{h}(\psi_2)(x) = \bar{h}(\psi_2,x)$. On the other hand, in the subcase where $\bar{h}(\psi_2,x)<\bar{h}(\psi_1,x)<1$ and there exists $z>x$ such that $\bar{h}(\psi_2,z)<1$, we have that $\hat{h}(\psi_1)(z)\to\hat{h}(\psi_2)(z)=1\to\hat{h}(\psi_2)(z)=\hat{h}(\psi_2)(z)\neq 1$, so $\Box (\hat{h}(\psi_1)\to\hat{h}(\psi_2))=0=\bar{h}(\psi_2,x)$. In the only remaining case, we have $x<y$. Then $\hat{h}(\psi_2)(x)=0$ and $\hat{h}(\psi_1)(y)\to\hat{h}(\psi_2)(y)\neq 1$, so $\Box(\hat{h}(\psi_1)\to\hat{h}(\psi_2))(x)=0=\hat{h}(\psi_2)(x)$. It follows that $\Box(\hat{h}(\psi_1)\to\hat{h}(\psi_2))(x)=\bar{h}(\psi_1\to\psi_2,x)$ in all cases, which proves (1).

For (2), observe that $\varphi$ is valid in the frame $(X,\leq,\{{\m A}_x : x\in X\})$ if and only if $\hat{h}(\varphi)(x)=1$ for every valuation $h$ in the frame and every $x\in X$, and by (1) this is equivalent to $\bar{h}(\varphi,x)=1$ for every valuation $h$ in the frame and every $x\in X$. Proposition~\ref{prop:temporal assign and val} shows that this holds if and only if $\varphi$ is valid in the temporal flow $(X,\leq,L)$, giving the result.
\end{proof}

\section{Soundness and completeness}\label{sec:soundness and completeness}
We next establish several soundness and completeness results for the relational semantics introduced in Section~\ref{def:frame}.

\begin{definition}
Let $\sf V$ be a variety of residuated lattices and let $\sf K$ be a class of frames.
\begin{enumerate}[\normalfont (1)]
\item We say that ${\sf K}$ is \emph{sound} for $\sf V$ if for any term $\varphi$, we have that if $\varphi\eq 1$ is valid in {\sf V} then $\sf K\km\varphi$.
\item We say that ${\sf K}$ is \emph{complete} for $\sf V$ if for any term $\varphi$, we have that $\sf K\km\varphi$ implies that $\varphi\eq 1$ is valid in $\sf V$.
\end{enumerate}
\end{definition}

The following lemma provides a sufficient condition for a class of frames to be sound and complete for a variety $\sf V$ of residuated lattices. All of our subsequent soundness and completeness results are derived from this general condition.

\begin{lemma}\label{thm:general condition}
Let $\sf V$ be a variety of residuated lattices generated by a class $\sf S$ of algebras in $\sf V$, and let $\sf K$ be a class of frames.
\begin{enumerate}[\normalfont (1)]
\item Suppose that $P(F)\in\sf V$ for all $F\in\sf K$. Then $\sf K$ is sound for $\sf V$.
\item Suppose that for each ${\m A}\in {\sf S}$ there exists $F\in\sf K$ such that $\m A$ embeds in $P(F)$. Then $\sf K$ is complete for $\sf V$.
\end{enumerate}
Consequently, if $\sf K$ satisfies the hypotheses of (1) and (2), then $\sf K$ is sound and complete for $\sf V$.
\end{lemma}

\begin{proof}
(1) Suppose that $\varphi\eq 1$ is valid in $\sf V$, and let $F\in\sf K$. Since $\varphi\eq 1$ is valid in $\sf V$, we have in particular that $\varphi\eq 1$ is valid in $P(F)\in\sf V$. By Lemma \ref{lem:forcing and algebra}, it follows that $F\km\varphi$. Hence $\sf K\km\varphi$, and $\sf K$ is sound for $\sf V$.

(2) Suppose that $\varphi\eq 1$ is not valid in $\sf V$. Then there exists ${\m A}\in\sf S$ such that $\varphi\eq 1$ is refuted in $\m A$. By hypothesis, there exists $F\in\sf K$ such that $\m A$ embeds in $P(F)$. Since the validity of equations is preserved by taking subalgebras, it follows that $\varphi\eq 1$ is refuted in $P(F)$. Hence there exists an algebraic assignment $h\colon\var\to P(F)$ such that $h(\varphi)\neq 1$. The map $h$ is a valuation in $F$, and $(F,h)\not\km\varphi$ by construction. This shows that $\sf K\not\km\varphi$, so by taking the contrapositive we get that $\sf K$ is complete for $\sf V$.
\end{proof}

The following theorem illustrates some sample applications of the foregoing lemma; the list is far from exhaustive.

\begin{theorem}\label{thm:zoo}
\begin{enumerate}[\normalfont (1)]
\item Each of the following classes of frames is sound and complete for ${\sf GBL}$.
\begin{enumerate}
\item The class of $\mv$-valued frames.
\item The class of frames valued in finite simple MV-algebra chains.
\end{enumerate}
\item Each of the following classes of frames is sound and complete for $\sf BL$.
\begin{enumerate}
\item The class of $\mv$-valued frames based in root systems.
\item The class of $\{\textbf{\L}_k : k\geq 2\}$-valued frames based in root systems.
\item The class of $\{\textbf{\L}_k : k\geq 2\}$-valued frames based in finite chains.
\end{enumerate}
\item For each $k\geq 2$, the class of frames valued in finite simple $k$-potent MV-algebra chains is sound and complete for ${\sf GBL}^k$.
\item For each $k\geq 2$, the class of frames valued in finite simple $k$-potent MV-algebra chains and based in root systems is sound and complete for ${\sf BL}^k$.
\item The class of $\bf 2$-valued frames is sound and complete for $\sf HA$.
\item The class of $\bf 2$-valued frames based in root systems is sound and complete for $\sf GA$.
\end{enumerate}
\end{theorem}

\begin{proof}
We first prove the soundness claims. In each case, we apply Lemma \ref{thm:general condition}(1) together with a result stating that $P(F)$ lies in the specified variety for each frame $F$ in specified class. For 1(a) and 1(b), the result follows from Lemma~\ref{lem:closure under poset prod}(2); for 2(a), 2(b), and 2(c) it follows from Lemma~\ref{lem:closure under poset prod}(5); item (3) follows from Lemma~\ref{lem:closure under poset prod}(2) and Lemma~\ref{lem:closure under poset prod}(3); item (4) follows from Lemma~\ref{lem:closure under poset prod}(3) and Lemma~\ref{lem:closure under poset prod}(5); item (5) follows from Lemma~\ref{lem:closure under poset prod}(4); and item (6) follows from Lemma~\ref{lem:closure under poset prod}(6). This proves all of the claimed soundness results.

For the completeness claims, in each case we give a generating class $\sf S$ for the specified variety such that each ${\m A}\in\sf S$ embeds in $P(F)$ for some $F$ in the specific class of frames. The completeness result then follows from Lemma~\ref{thm:general condition}(2) in each case. For (1), take ${\sf S}$ to be the class of almost finite GBL-algebras. Each ${\m A}\in\sf S$ embeds in a poset product of finite simple MV-algebra chains by Lemma~\ref{lem:embed almost finite}, giving 1(b). 1(a) then follows from Lemma~\ref{lem:subalgebras of poset products} since each finite simple MV-algebra chain $\textbf{\L}_k$ is a subalgebra of $\mv$.

For (2), take ${\sf S}$ to be the class of finite BL-algebra chains. This class generates ${\sf BL}$, and for each finite BL-algebra chain ${\m A}$ we have that $\Delta({\m A})$ is a chain. Hence Lemma~\ref{lem:embed almost finite} gives 2(c), and \emph{a fortiori} 2(b). 2(a) follows from Lemma~\ref{lem:subalgebras of poset products} as well.

(3) follows immediately from Lemma~\ref{lem:embed almost finite} by taking $\sf S$ to be ${\sf GBL}^k$. (4) follows likewise by taking ${\sf S}$ to be ${\sf BL}^k$, and noting that the poset of values of any $k$-potent BL-algebra is a root system. (5) and (6) follow by similar considerations by noting that the minimal nontrivial deductive filter of any subdirectly irreducible Heyting algebra is $\m 2$.
\end{proof}
In light of the discussion in Section~\ref{sec:comparison}, Theorem~\ref{thm:zoo}(1)(a) is the main result of \cite{LOR2020} (therein Theorems 3.13 and 3.14) and Theorem~\ref{thm:zoo}(2) is the main result of \cite{ABM2009} (therein Theorem 1.1). \cite{LOR2020} notes that its results specialize to Kripke semantics for intuitionistic logic (Theorem 3.10 therein, Theorem~\ref{thm:zoo}(5) here), while \cite{ABM2009} specializes its results to G\"odel-Dummett logic (Proposition 3.3 therein, Theorem~\ref{thm:zoo}(6) here). Both our general recipe for obtaining results of this kind and the results for ${\sf GBL}^k$ and ${\sf BL}^k$ (Theorem~\ref{thm:zoo}(3,4)) are new.

\section{Logics characterized by frames}\label{sec:afmp}

It is well-known that some extensions of intuitionistic logic are not characterized by any class of {\bf 2}-valued frames, a thoroughly-studied phenomenon usually called \emph{Kripke incompleteness}. Admitting frames valued in richer classes of residuated lattices considerably expands the expressive power of the relational semantics, but strips away some of its simplicity. The goal of this section is to offer a preliminary study of the expressivity of frames valued in finite MV-algebra chains. In particular, we provide some sufficient conditions for an axiomatic extension of the logic of GBL-algebras to be characterized by frames valued in $\{\textbf{\L}_k : k\geq 2\}$. Our first result concerns logics with the finite model property.

\begin{proposition}\label{thm:fmp}
Let $\sf V$ be a subvariety of $\sf GBL$ with the finite model property. Then there exists a class of $\{\textbf{\L}_k : k\geq 2\}$-valued frames that is sound and complete for $\sf V$.
\end{proposition}

\begin{proof}
Let $\sf S$ be the class of finite member of $\sf V$, which by hypothesis generates $\sf V$. By Lemma~\ref{lem:finite rep} we have that for each ${\m A}\in{\sf S}$ there exists a family ${\m A}_x\in\{\textbf{\L}_k : k\geq 2\}$, $x\in\Delta=\Delta({\m A})$, such that
$${\m A}\cong\prod_{(\Delta,\subseteq)} {\m A}_x.$$
For each ${\m A}\in\sf S$, define $F({\m A})=(\Delta({\m A}), \subseteq, \{{\m A}_x : x\in\Delta({\m A})\})$ and ${\sf K} = \{F({\m A}) : {\m A}\in{\sf S}\}$. Then ${\m A}\cong P(F({\m A}))$, so in particular $P(F)\in\sf V$ for each $F\in{\sf K}$. Thus Lemma~\ref{thm:general condition}(1) gives that $\sf K$ is sound for $\sf V$. Since each ${\m A}\in\sf S$ embeds in $P(F({\m A}))$, Lemma~\ref{thm:general condition}(2) also gives that $\sf K$ is complete for $\sf V$.
\end{proof}

The situation is more complicated for varieties with the AFMP. Recall, if $\m A$ is an almost finite GBL-algebra, then $\m A$ embeds into the poset product $E({\m A})=\prod_{(\Delta,\subseteq)} {\m A}_x,$ where $\Delta=\Delta({\m A})$ is the set of values of $\m A$ and ${\m A}_x$ is the minimum nontrivial deductive filter of ${\m A}/x$ for each $x\in\Delta({\m A})$. Using this notation, for each almost finite GBL-algebra ${\m A}$ we define a frame
$$F({\m A}) = (\Delta({\m A}),\subseteq,\{{\m A}_x : x\in\Delta({\m A})\}),$$
and let $F({\sf V})=\{F({\m A}) : {\m A}\in{\sf V}\text{ is almost finite}\}$.

\begin{proposition}\label{prop:completeness almost finite}
Let $\sf V$ be a subvariety of $\sf GBL$ with the almost finite model property. Then $F({\sf V})$ is complete for $\sf V$.
\end{proposition}

\begin{proof}
By hypothesis, ${\sf V}$ is generated by its class of almost finite members. By Theorem~\ref{lem:embed almost finite}, if ${\m A}\in{\sf V}$ is almost finite, then ${\m A}$ embeds in $E({\m A})$. Observe that $E({\m A})=P(F({\m A}))$ by construction, so in particular ${\m A}$ embeds in $P(F({\m A}))$. The result then follows from Lemma~\ref{thm:general condition}(2).
\end{proof}

Soundness is a more difficult issue for varieties with the AFMP. To express our soundness result, we recall the substructural hierarchy of Ciabattoni, Galatos, and Terui \cite{CGT2008, CGT2012}. We give the variant of this defined in \cite{Fr2016}.
\begin{definition}
For $n\geq 0$, recursively define sets $\mathcal{P}_n$ and $\mathcal{N}_n$ of terms as follows:
\begin{enumerate}[\normalfont (P1)]
\item[(0)] $\mathcal{P}_0=\mathcal{N}_0=\var$.
\item[(P1)] $1$ and all terms of $\mathcal{N}_n$ belong to $\mathcal{P}_{n+1}$.
\item[(P2)] If $t,u\in\mathcal{P}_{n+1}$, then $t\join u,t\cdot u\in\mathcal{P}_{n+1}$.
\item[(N1)] $0$ and all terms of $\mathcal{P}_n$ belong to $\mathcal{N}_{n+1}$.
\item[(N2)] If $t,u\in\mathcal{N}_{n+1}$, then $t\meet u\in\mathcal{N}_{n+1}$.
\item[(N3)] If $t\in\mathcal{P}_{n+1}$ and $u\in\mathcal{N}_{n+1}$, then $t\to u\in\mathcal{N}_{n+1}$.
\end{enumerate}
Furthermore, let $\mathcal{P}_2^*$ be the smallest set of terms such that:
\begin{enumerate}[\normalfont ($P1^*$)]
\item[(P1)$^*$] $\mathcal{P}_2\subseteq \mathcal{P}_2^*$.
\item[(P2)$^*$] If $t,u\in\mathcal{P}_2^*$, then $t\meet u,t\join u, t\cdot u\in\mathcal{P}_2^*$.
\item[(P3)$^*$] If $u\in\mathcal{P}_2^*$ and $t\in\mathcal{P}_1$, then $t\to u\in\mathcal{P}_2^*$.
\end{enumerate}
Finally, define $\mathcal{N}_2^*$ as $\mathcal{N}_2$ is defined, but replacing (N3) by
\begin{enumerate}[\normalfont ($P1^*$)]
\item[(N3)$^*$] If $t\in\mathcal{P}_2^*$ and $u\in\mathcal{N}_2^*$, then $t\to u\in\mathcal{N}_2^*$.
\end{enumerate}
\end{definition}

Every term in the language of residuated lattices belongs to some $\mathcal{P}_n$ and $\mathcal{N}_n$. \cite{CGT2012} shows that every equation of the form $t\eq 1$, where $t\in\mathcal{N}_2$, can be expressed by structural rules in a corresponding sequent calculus, and Je\v{r}\'abek has shown that every variety of commutative residuated lattices is axiomatizable by equations of the form $t\eq 1$, where $t\in\mathcal{N}_3$. We will apply the following preservation result.

\begin{lemma}[{\cite[Theorem 5.9]{Fr2016}}]\label{lem:preservation under conuclei}
Let $\m A$ be a residuated lattice and let $\sigma$ be a conucleus on $\m A$. Suppose that $t\in\mathcal{P}_2^*$ and $u\in\mathcal{N}_2^*$. Then if $\m A$ satisfies $t\leq u$, so does ${\m A}_\sigma$.
\end{lemma}

Call an equation \emph{conuclear} if it is of the form $t\to u\eq 1$, where $t\in\mathcal{P}_2^*$ and $u\in\mathcal{N}_2^*$.

\begin{theorem}
Suppose that $\sf V$ is a subvariety of $\sf GBL$ with the almost finite model property. Assume further that $\sf V$ is axiomatized relative to $\sf GBL$ by a set $\Sigma$ of equations in the signature $\{\meet,\join,\cdot,\to,1\}$, and that each equation in $\Sigma$ either (1) does not contain $\to$, or (2) is conuclear. Then $F({\sf V})$ is sound and complete for $\sf V$.
\end{theorem}

\begin{proof}
Proposition \ref{prop:completeness almost finite} gives completeness. For soundness, let ${\m A}$ be an almost finite member of ${\sf V}$. Then ${\m A}$ satisfies $\Sigma$. If $x$ is a value of ${\m A}$, then ${\m A}_x$ is a $\{\meet,\join,\cdot,\to,1\}$-subalgebra of the quotient ${\m A}/x$, hence ${\m A}_x$ satisfies $\Sigma$ for each $x\in\Delta=\Delta({\m A})$. It follows that the direct product ${\m B} = \prod_{x\in\Delta} {\m A}_x$ also satisfies $\Sigma$. By definition, $P(F({\m A}))$ is ${\m B}_{\Box}$. Each identity in $\Sigma$ is preserved under $\Box$ by Lemmas~\ref{lem:product subalgebra} and~\ref{lem:preservation under conuclei}, so it follows that $P(F({\m A}))$ satisfies $\Sigma$. Hence $P(F({\m A}))\in\sf V$ for each almost finite ${\m A}\in{\sf V}$, and thus it follows from Theorem~\ref{thm:general condition} that $F({\sf V})$ is sound for ${\sf V}$.
\end{proof}



\AuthorAdressEmail{Wesley Fussner}{Laboratoire J.A Dieudonn\'e\\
Universit\'e C\^ote d'Azur \\
Parc Valrose 06108 Nice Cedex 2 \\
France}{wfussner@unice.fr}

\end{document}